\documentclass[11pt]{amsart}

\usepackage[margin=1in]{geometry}
\usepackage{amsmath,amssymb,amsthm,mathtools}
\usepackage[hidelinks]{hyperref}
\usepackage{microtype}
\usepackage{orcidlink}

\newtheorem{theorem}{Theorem}
\newtheorem{lemma}[theorem]{Lemma}
\newtheorem{remark}[theorem]{Remark}
\newcommand{\F}{\mathbb F}
\newcommand{\I}{\mathcal I}

\title[Increasing Unrefinable Subgroup Chains]{Finite Soluble Groups Need Not
Admit Increasing Unrefinable Subgroup Chains}
\author[Richie Sater]{Richie Sater\,\orcidlink{0009-0007-9051-8207}}
\address{Independent Researcher, United States}
\email{richiesater@gmail.com}
\date{August 11, 2026}

\begin{document}
\begin{abstract}
In correspondence beginning on August 8, 2026, V.~S.~Monakhov and
I.~L.~Sokhor communicated the following question to the author: does every
finite group have an unrefinable chain of subgroups whose successive indices
are nondecreasing?
We answer this question in the negative, even among soluble groups.  For
every odd prime $p$, we construct a soluble matrix group
$G_p\leq \operatorname{GL}_5(p)$ of order $2^6p^8$ with no such chain.
These counterexamples form an infinite family whose smallest member,
$G_3$, has order $419904$.  The construction is based on an example due to
Kohler, and the obstruction follows from three elementary calculations of
maximal-subgroup indices.
\end{abstract}
\maketitle

\section{Introduction}

Let $G$ be a nontrivial finite group.  An unrefinable chain of subgroups
\[
  1=G_0<G_1<\cdots<G_n=G
\]
is one in which $G_{i-1}$ is maximal in $G_i$ for every $i$.  Associate
with this chain the indices $j_i=\lvert G_i:G_{i-1}\rvert$.  We call the
chain increasing if
\[
  j_1\leq j_2\leq\cdots\leq j_n.
\]
Monakhov and Sokhor call such a chain a $(<)$-chain~\cite{MS2025}.
In correspondence beginning on August 8, 2026, V.~S.~Monakhov and
I.~L.~Sokhor communicated to the author the following question, stated here
in equivalent terminology.\footnote{The correspondence is the source through
which the author learned of the question; no claim is made here about the
problem's original formulation.  Reference~\cite{MS2025} is cited for
terminology and background, not as a published source of the
universal-existence question.}
\begin{quote}
Does every finite group contain an increasing unrefinable subgroup chain?
\end{quote}
We answer this question in the negative.

\begin{theorem}\label{thm:main}
For every odd prime $p$, there is a soluble group $G_p$ of order $2^6p^8$
which has no increasing unrefinable subgroup chain.
\end{theorem}

The smallest group in the family has order $2^6 3^8=419904$.  We do not
claim that it is the smallest counterexample.  Our construction is based on
the $t=2$ case of Kohler's work on maximal-subgroup indices
\cite{Kohler1964}.  Kohler used this construction to show that bounds on
the indices of maximal subgroups of a finite group need not be inherited by
its subgroups.  In the matrix model below, the four coordinates occurring in
$x$ and $y$ correspond to four degree-one generators, while the four
entries of $z$ record the independent cross-commutators.  The resulting
$p$-group therefore has order $p^{4+4}=p^8$.  We exploit the resulting
nested pattern of maximal-subgroup indices.

\section{A lemma on maximal subgroups}

For a finite group $A$, write
\[
  \I(A)=\{\lvert A:B\rvert \mid B\text{ is maximal in }A\}.
\]

\begin{lemma}\label{lem:semidirect}
Let $V$ be an elementary abelian $p$-group, let $H$ be a $p'$-group acting
on $V$, put $K=V\rtimes H$, and let $\pi:K\to H$ be the natural
projection.  Every maximal subgroup of $K$ has one of the following forms:
\begin{enumerate}
\item $\pi^{-1}(J)=V\rtimes J$, where $J$ is maximal in $H$;
\item a $V$-conjugate of $W\rtimes H$, where $W$ is a maximal proper
      $H$-submodule of $V$.
\end{enumerate}
\end{lemma}

\begin{proof}
Let $T$ be maximal in $K$.  If $V\leq T$, then $T/V$ is maximal in
$K/V\cong H$, which gives the first form.  Suppose $V\nleq T$.  Then
$VT=K$.  Put $W=V\cap T$.
The group $T$ normalizes $W$, while $V$ centralizes $W$; hence $W\lhd K$.
In $K/W$, both $T/W$ and $WH/W$ are complements to $V/W$.  By the
Schur--Zassenhaus theorem, these complements are conjugate by an element of
$V/W$.  Lifting this conjugacy shows that $T$ is $V$-conjugate to
$W\rtimes H$.  Moreover, maximality of $T$ is equivalent to maximality of
$W$ among the proper $H$-submodules of $V$.
\end{proof}

In particular, if $H$ is a $2$-group, the subgroups of the first type have
index $2$.  If $V$ is completely reducible, the indices of the subgroups of
the second type are the orders of the irreducible quotients of $V$.

\section{The counterexample family}

Fix an odd prime $p$ and put $F=\F_p$.  Let
\[
 s=\begin{pmatrix}0&1\\1&0\end{pmatrix},\qquad
 t=\begin{pmatrix}1&0\\0&-1\end{pmatrix},\qquad
 D=\langle s,t\rangle\leq\operatorname{GL}_2(F).
\]
Then $D\cong D_8$, where $D_8$ has order $8$.  The natural $FD$-module
$U=F^2$ is absolutely irreducible.  Indeed, after extending scalars to an
algebraic closure, the only $t$-invariant lines are the two coordinate
lines, and $s$ interchanges them.

Let $H=D\times D$.  Define
\[
 L=\left\{
 \ell(x,y,z)=
 \begin{pmatrix}
 I_2&x&z\\
 0&1&y\\
 0&0&I_2
 \end{pmatrix}\mathrel{:}
 x\in F^{2\times1},\ y\in F^{1\times2},\ z\in F^{2\times2}
 \right\}.
\]
Here the block sizes are $2,1,2$.  Direct multiplication yields
\begin{equation}\label{eq:mult}
 \ell(x,y,z)\ell(x',y',z')
 =\ell(x+x',y+y',z+z'+xy').
\end{equation}
With the convention $[a,b]=a^{-1}b^{-1}ab$, it follows that
\begin{equation}\label{eq:commutator}
 [\ell(x,y,z),\ell(x',y',z')]
 =\ell(0,0,xy'-x'y).
\end{equation}
Thus $\lvert L\rvert=p^8$.  Put
\[
 Z=\{\ell(0,0,z)\},\quad
 X=\{\ell(x,0,0)\},\quad
 Y=\{\ell(0,y,0)\}.
\]
The group $Z$ is central in $L$.  The outer products $xy$ span
$F^{2\times2}$, so \eqref{eq:commutator} gives $L'=Z$.  Since $p$ is odd,
$L$ has exponent $p$; explicitly,
\[
 \ell(x,y,z)^p
 =\ell\left(0,0,pz+\binom p2xy\right)=1.
\]
The standard formula for the Frattini subgroup of a finite $p$-group now
gives
\begin{equation}\label{eq:frattini}
 \Phi(L)=L^pL'=Z.
\end{equation}

Embed $H$ in $\operatorname{GL}_5(F)$ by identifying $(A,B)\in D\times D$
with the block-diagonal matrix $\operatorname{diag}(A,1,B)$.  Conjugation
by this matrix acts by
\[
 x\mapsto Ax,\qquad y\mapsto yB^{-1},\qquad z\mapsto AzB^{-1}.
\]
Consequently, $H$ normalizes both $L$ and $Z$.  Define
\[
 G_p=L\rtimes H\leq\operatorname{GL}_5(p).
\]
Then $Z\lhd G_p$ and
$\lvert G_p\rvert=p^8\cdot 8^2=2^6p^8$.  Since both $L$ and $H$ are
soluble, so is $G_p$.

Set $\overline X=XZ/Z$ and $\overline Y=YZ/Z$.  As an $H$-module,
\begin{equation}\label{eq:modules}
 L/Z=\overline X\oplus\overline Y,
\end{equation}
where $\overline X$ and $\overline Y$ are nonisomorphic irreducible modules
of dimension $2$.  More explicitly, the first factor of $D\times D$ acts
naturally on $\overline X$, while the second acts trivially.  The second
factor acts dually on $\overline Y$, while the first acts trivially.  Thus
\[
 C_H(\overline X)=1\times D,
 \qquad
 C_H(\overline Y)=D\times 1.
\]
These action kernels also show that the two modules are nonisomorphic.
Since $p\nmid \lvert H\rvert=64$, Maschke's theorem applies.  Hence the
only maximal $H$-submodules of $L/Z$ are $\overline X$ and $\overline Y$.

Moreover,
\begin{equation}\label{eq:tensor}
 Z\cong U\boxtimes U^*
\end{equation}
is the external tensor product for the two direct factors of $H$.  Hence
$Z$ is an irreducible $H$-module of dimension $4$.  Indeed, over an
algebraic closure $\overline F$, both $U\otimes_F\overline F$ and
$U^*\otimes_F\overline F$ remain irreducible, and their external tensor
product is irreducible for $D\times D$.

\section{The index obstruction}

\begin{lemma}\label{lem:spectra}
The maximal-index spectra are
\begin{align*}
 \I(G_p)&=\{2,p^2\},\\
 \I(M_X)=\I(M_Y)&=\{2,p^2,p^4\},\\
 \I(N)&=\{2,p^4\},
\end{align*}
where
\[
 M_X=(Z\times X)\rtimes H,\qquad
 M_Y=(Z\times Y)\rtimes H,\qquad
 N=Z\rtimes H.
\]
Every maximal subgroup of $G_p$ of index $p^2$ is $G_p$-conjugate to $M_X$
or $M_Y$.  Every maximal subgroup of $M_X$ of index $p^2$ is
$M_X$-conjugate to $N$, and every maximal subgroup of $M_Y$ of index $p^2$
is $M_Y$-conjugate to $N$.
\end{lemma}

\begin{proof}
First, we show that every maximal subgroup $M$ of $G_p$ contains $Z$.  If
$Z\nleq M$, then maximality gives $G_p=ZM$, and the modular law together
with \eqref{eq:frattini} gives
\[
 L=L\cap ZM=Z(L\cap M)=\Phi(L)(L\cap M).
\]
Recall the standard Frattini fact that $L=\Phi(L)K$ for a subgroup $K\leq L$
implies $K=L$: otherwise, a maximal subgroup of $L$ containing $K$ would
contain both $K$ and $\Phi(L)$, and hence all of $L$, a contradiction.
Applying this fact with $K=L\cap M$ forces $L\cap M=L$, so $Z\leq M$,
again a contradiction.

It follows that the maximal subgroups of $G_p$ are the inverse images of
the maximal subgroups of
\[
 G_p/Z=(\overline X\oplus\overline Y)\rtimes H.
\]
Apply Lemma~\ref{lem:semidirect} and \eqref{eq:modules}.  The maximal
subgroups of the first type have index $2$, while those of the second type
have index $p^2$.  This proves the first equality and the assertion
concerning $M_X$ and $M_Y$.

In $M_X$, the normal elementary abelian subgroup $Z\times X$ is the direct
sum of nonisomorphic irreducible $H$-modules of dimensions $4$ and $2$.
Its only maximal $H$-submodules are therefore $Z$ and $X$.  Another
application of Lemma~\ref{lem:semidirect} gives the indices $2,p^2,p^4$;
the index-$p^2$ case retains $Z$ and gives an $M_X$-conjugate of $N$.
The same argument applies to $M_Y$.  Finally, $Z$ is irreducible of
dimension $4$, so Lemma~\ref{lem:semidirect} applied to $N=Z\rtimes H$
gives $\I(N)=\{2,p^4\}$.
\end{proof}

\begin{proof}[Proof of Theorem~\ref{thm:main}]
Suppose
\[
 1=G_0<G_1<\cdots<G_n=G_p
\]
is unrefinable and its indices $j_i=\lvert G_i:G_{i-1}\rvert$ are
nondecreasing.
By Lemma~\ref{lem:spectra}, $j_n\in\{2,p^2\}$.  We cannot have $j_n=2$:
then monotonicity would force every $j_i$ to be $2$, whereas
$p\mid\lvert G_p\rvert$.  Thus $j_n=p^2$, and $G_{n-1}$ is
$G_p$-conjugate to $M_X$ or $M_Y$.

Now $j_{n-1}\leq p^2$.  The middle spectrum in
Lemma~\ref{lem:spectra} shows that $j_{n-1}$ is $2$ or $p^2$.  It cannot be
$2$, by the same divisibility argument applied to $G_{n-1}$.  Hence
$j_{n-1}=p^2$, and $G_{n-2}$ is $G_p$-conjugate to $N$.

Finally, $j_{n-2}\leq p^2$.  Since $G_{n-2}$ is $G_p$-conjugate to $N$ and
$\I(N)=\{2,p^4\}$, we must have $j_{n-2}=2$.  This forces all preceding
indices to equal $2$, which is impossible because $p\mid\lvert N\rvert$.
The contradiction proves the theorem.
\end{proof}

\section{The fully explicit group \texorpdfstring{$G_3$}{G3}}

To facilitate independent verification, we give the smallest member of the
family without parameters.  All matrices below are over $\F_3$, so $-1=2$.
Let $E_{ij}$ denote the standard matrix unit, write
$e_{ij}=I_5+E_{ij}$, and set
\[
 \begin{array}{ll}
 a_s=\operatorname{diag}(s,1,I_2),&
 a_t=\operatorname{diag}(t,1,I_2),\\
 b_s=\operatorname{diag}(I_2,1,s),&
 b_t=\operatorname{diag}(I_2,1,t).
 \end{array}
\]
Then the concrete example is the following subgroup of
$\operatorname{GL}_5(3)$:
\begin{equation}\label{eq:explicit-g3}
 G_3=\langle e_{13},e_{23},e_{34},e_{35},a_s,a_t,b_s,b_t\rangle.
\end{equation}
Define
\[
 \begin{split}
 Z_3&=\langle e_{14},e_{15},e_{24},e_{25}\rangle,\\
 H_3&=\langle a_s,a_t,b_s,b_t\rangle,\\
 M_{X,3}&=\langle Z_3,e_{13},e_{23},H_3\rangle,\\
 M_{Y,3}&=\langle Z_3,e_{34},e_{35},H_3\rangle,
 \qquad N_3=\langle Z_3,H_3\rangle.
 \end{split}
\]
Here $\lvert H_3\rvert=64$.  The orders and complete maximal-index spectra
relevant to the obstruction are
\[
\begin{array}{c|c|c}
 A&\lvert A\rvert&\I(A)\\ \hline
 G_3&419904&\{2,9\}\\
 M_{X,3},M_{Y,3}&46656&\{2,9,81\}\\
 N_3&5184&\{2,81\}
\end{array}
\]
The standalone file \texttt{G3-verification.g} accompanying this manuscript
defines all eight generators in \eqref{eq:explicit-g3} as explicit
$5\times5$ matrices.  Using only standard GAP commands, it verifies the
group and subgroup orders, the five named maximal inclusions, every
conjugacy class of maximal subgroups and its index, the coverage of all
relevant classes by the named subgroups, and the result of an exhaustive
top-down search for an increasing chain.  The file has been run unchanged
under GAP 4.11.1 and GAP 4.16.0.
Thus the computation can be checked from the attached file without access
to an online repository.

\begin{remark}\label{rem:computation}
GAP computations for $p=3,5,7$ confirm the predicted spectra; for
$p=3$ these are
\[
 \{2,9\},\qquad \{2,9,81\},\qquad \{2,81\},
\]
and an exhaustive branch search finds no increasing unrefinable chain.  These
computations corroborate the theorem but are not used in its proof.  The
standalone $p=3$ file, complete code, tests, and captured outputs are
publicly available in release
\href{https://github.com/RichieSater/monotone-maximal-chains/releases/tag/v1.1.0}{v1.1.0}
of the accompanying repository~\cite{Repository}.
\end{remark}

\section*{Acknowledgments}

The author thanks V.~S.~Monakhov and I.~L.~Sokhor for bringing the question
to his attention through their correspondence.

\section*{Code and data availability}

The complete computational supplement is publicly available in the
\href{https://github.com/RichieSater/monotone-maximal-chains}{verification
repository}.  Release \texttt{v1.1.0} is the version corresponding to this
manuscript.  It includes the standalone verifier and captured outputs from
GAP 4.11.1 and GAP 4.16.0.  Source code is licensed under the MIT License,
and the manuscript and documentation are licensed under CC BY 4.0.

\section*{Declaration of generative AI and AI-assisted technologies}

During this project the author used Anthropic's Claude and OpenAI Codex for
mathematical exploration and criticism, literature-search assistance,
drafting and debugging GAP code, and manuscript organization and editing.
The author independently checked the mathematical argument, reran all tests
in the repository, and verified the cited references.  He takes full
responsibility for the content of the article.

\end{document}